\documentclass{amsart}

\usepackage{graphicx}        
\usepackage[bottom]{footmisc}
\usepackage{url}
\usepackage[utf8]{inputenc}
\usepackage{amsmath}
\usepackage{amsthm}
\usepackage{amssymb}
\usepackage{physics}
\newtheorem{exmp}{Example}[section]
\newtheorem{assumption}{Assumption}

\newtheorem{theorem}{Theorem}
\newtheorem{lemma}[theorem]{Lemma}

\begin{document}

\title[Splitting approximations for SDDEs]{On splitting strategies for the numerical solution of stochastic delay differential equations with correlated noises}
\author[C. Kelly]{C\'onall Kelly}
\address{School of Mathematical Sciences, University College Cork, Western Gateway Building, Western Road, Cork, Ireland.}
\email{conall.kelly@ucc.ie}

\author[W. Tang]{Wenshi Tang}
\address{School of Mathematical Sciences, University College Cork, Western Gateway Building, Western Road, Cork, Ireland.}
\email{123121924@umail.ucc.ie}

%
%

\begin{abstract}
In this article we investigate the numerical solution of a scalar semilinear stochastic delay differential equation (SDDE) where the linear instantaneous feedback and nonlinear delayed feedback terms are perturbed by a pair of standard Brownian motions with correlation $\rho$. Such SDDEs may be naturally decomposed into two subsystems: a linear stochastic differential equation (SDE) without delay, and a nonlinear SDDE. 

Splitting methods work by solving each subsystem separately and composing the results over a single step. Our main theoretical result provides a bound on the mean-square error of a particular strategy for doing this, known as Lie-Trotter splitting. This bound implies that the method is mean-square strongly convergent with order $1/2$ when $\rho=0$, so that the noises are uncorrelated, but assurances of convergence are lost when $\rho\neq 0$. Indeed we develop an upper bound on the global mean-square error with a term that depends linearly on the magnitude of the correlation, and is independent of the stepsize. 

While our theoretical error bound is an estimate from above, we conduct numerical experiments that confirm the order of mean-square strong convergence of Lie-Trotter splitting in the $\rho=0$ case, and demonstrate a rapid fall-off to effectively zero as $|\rho|$ increases. Similar numerical results are observed for an alternative commonly used strategy known as Strang splitting. Nonetheless, by carefully reorganising the subsystems into which we split the SDDE, we can improve the range of values of $\rho$ over which a nonzero order of convergence is observed numerically.
\end{abstract}


\maketitle

\section{Introduction}

\subsection{General background: stochastic differential equations (SDEs)}
Differential equations are used to model systems that evolve dynamically according to one or more feedback mechanisms. Real-world systems such as financial markets are subject to uncertainty and it is necessary to incorporate the possibility of random effects disturbing the system. In this work, we model the random perturbations by Brownian motion processes, leading to an SDE of It\^o-type. A one-dimensional linear SDE perturbed by a single scalar standard Brownian motion $B$ may be written in the first instance as an integral equation 
\begin{equation}\label{eq:linSDEmotivate}
X(t)=X_0+\int_0^t \mu X(s)ds+\int_0^t \sigma X(s)dB(s),\quad t\in[0,T],
\end{equation}
for some constants $\mu,\sigma\in\mathbb{R}$ and initial value $X(0)=X_0\in\mathbb{R}$. The second integral on the RHS of \eqref{eq:linSDEmotivate} is taken with respect to the Brownian motion, and is known as an It\^o integral. Equation \eqref{eq:linSDEmotivate} is said to have stochastic differential form when it is written 
\begin{equation}\label{eq:dx}
dX(t)=\mu X(t)dt+\sigma X(t)dB(t),\quad t\in[0,T].
\end{equation}
The SDE \eqref{eq:dx} is widely used to model asset prices in financial markets, and is a rare example of an SDE that admits an explicit solution (see Section \ref{sec:exact}). In the real-world form of this model, $\mu$ represents the constant underlying growth rate of an asset and $\sigma$ its constant volatility. Nonetheless, most SDEs cannot be solved explicitly, and so it is necessary to develop effective and efficient numerical solution techniques. An excellent technical introduction to It\^o integrals and SDEs may be found in Mao~\cite{Mao2008}, with a more applied treatment available in~\cite[Chapter 10]{ISM}. 

The SDE \eqref{eq:dx} has solutions that are Markovian. In many applications it is desirable to model non-Markovian dynamics. In finance, we may be motivated to model time-varying asset-price volatility with historical dependencies as a stochastic process. Alternatively we may be motivated by the fact that, despite classical assumptions about market efficiency, there is evidence that past prices of an asset may sometimes influence its future evolution: see the discussion in~\cite{AChunxiang2018,Arriojas2007}, and \cite{MaoSabanis2013}. Applications of SDEs to the modelling of population dynamics naturally prompt us to incorporate the effects of gestation delay (see~\cite{BaharMao2004}). An SDE which, like \eqref{eq:dx}, incorporates only instantaneous feedback will not capture these properties. 

\subsection{Numerical solution of stochastic delay differential equations (SDDEs)}
Motivated by this discussion, let us extend \eqref{eq:dx} to include additional nonlinear feedback terms with delay in both the drift and diffusion coefficients. Consider the scalar semilinear SDDE of It\^o-type 
\begin{align}
dX(t)
&= \left[\mu X(t)+f(X(t-\tau))\right]dt + \sigma X(t)dW_1(t)+g(X(t-\tau))dW_2(t),\quad t\in(0,T], \nonumber\\
X(t)
&= \psi(t), \quad t\in[-\tau,0],\label{eq:SDDEcorr}
\end{align}
with delay length $\tau>0$, and linear coefficients $\mu,\sigma\in\mathbb{R}$. An SDDE of the form \eqref{eq:SDDEcorr} describes a system with two feedback channels, one of which is linear and instantaneous, the other nonlinear and subject to feedback delay of length $\tau$. Both channels are noisy, and the noise perturbations, though not identical on each channel, may be interdependent. Therefore we suppose that $W=[W_1,W_2]^T$ is a correlated pair of standard Brownian motions such that 
\[
W(t)\sim \mathcal{N}\left(0,\Sigma\right),\quad \Sigma=t\cdot\begin{pmatrix}1 & \rho\\ \rho & 1\end{pmatrix},\quad \rho\in(-1,1).
\]
Due to the presence of both noise and feedback delays, SDDE solutions can oscillate about an equilibrium point. We say that a nontrivial continuous function $y:[0,\infty)\to\mathbb{R}$ is oscillatory (about zero) if the set 
\(
Z_y=\{t\geq 0\,:\,y(t)=0\}
\)
satisfies $\sup Z_y=\infty$. This definition may be applied pathwise to solutions of \eqref{eq:SDDEcorr}, which are said to be almost surely (a.s.) oscillatory if and only if there exists an a.s. event on which all individual trajectories are oscillatory functions. The joint role of noise and feedback delays in oscillation of SDDE solutions is explored in \cite{ApplebyBuckwar2005,appleby2004asymptotic,appleby2004oscillation,appleby2006NonlinearOscillation}, building on the analysis of delay-induced oscillation in deterministic differential equations in Gopalsamy~\cite{G1992}. 

In this article, we are interested in the efficient numerical solution of solutions of \eqref{eq:SDDEcorr} via splitting. Splitting methods work by decomposing an equation into two or more subequations, each of which may be solved separately over a single discretisation step in a manner most suited to its structure, before being recombined. Readers interested in a general introduction to numerical solution of SDEs may refer to Higham \& Kloeden~\cite{HigKloe2021}. Those seeking an introductory reference focused only on numerical solution of SDDEs may consult Buckwar~\cite{Buckwar2000}.

By Cholesky factorisation of the covariance matrix $\Sigma$ and by the linear transformation property of the Gaussian distribution, it is equivalent to sample from the SDDE
\begin{multline}\label{eq:SDDEind}
dX(t)= \left[\mu X(t)+f(X(t-\tau))\right]dt 
\\+ \sigma X(t)\left[\sqrt{1-\rho^2}\,dB_1(t)+\rho\,dB_2(t)\right]+g(X(t-\tau))dB_2(t),\quad t\in(0,T],
\end{multline}
where $B=[B_1,B_2]^T$ is a pair of independent standard Brownian motions, and therefore $B(t)\sim \mathcal{N}(0,\mathbb{I}_2t)$.
We are motivated by the observation that terms on the RHS of \eqref{eq:SDDEcorr} (and \eqref{eq:SDDEind}) may be naturally divided into linear terms without a feedback delay, and nonlinear terms with a feedback delay. This induces a natural decomposition of the SDDE into two subsystems: a linear SDE without delay that can be solved exactly, and a nonlinear SDDE, which must be solved by some suitably tailored SDDE approximation method. Our results indicate that care must be taken when following such a strategy, since the correlation of $W_1$ and $W_2$ can have significant and potentially detrimental effects on the rate at which a splitting method converges in the strong stochastic sense to the true solution of the SDDE.

In this article we will examine two specific discretisation strategies of splitting type that are distinguished by the manner in which they recombine these subsystems. They are referred to as Lie-Trotter and Strang splitting respectively. In the case of Lie-Trotter, we find that when $W_1$ and $W_2$ are independent sources of noise, so that $\rho=0$, mean-square strong convergence can be confirmed theoretically as being of order $1/2$. However, when $\rho\neq 0$ our theoretical bound of the strong RMS error, though uniform in the stepsize $\delta t$, is no longer guaranteed to converge to zero with $\delta t$. We do not present a theoretical error analysis for the Strang method in this article. 

Numerically we explore the effect of varying $\rho$ on the empirical order of strong RMS convergence for both Lie-Trotter and Strang splitting. For both methods, our observations are consistent with order-$1/2$ strong convergence when $\rho=0$, and we observe a rapid drop-off to zero in the convergence order as $|\rho|$ increases. By modifying the split to ensure that individual independent Brownian motions are not replicated across subsystems, we can increase the effective range of $\rho$ over which we observe a nonzero order of convergence.

\subsection{Structure of the article.}
In Section \ref{sec:exact}, we demonstrate how a variation of constants form of the unique global solutions of \eqref{eq:SDDEind} may be constructed via the method-of-steps. This is useful for motivating our proposed splitting strategies, characterising the pathwise error in the proof of our main theorems, and providing a reference solution for our numerical investigation. 

In Section \ref{sec:split} we introduce the Lie-Trotter and Strang splitting strategies applied to the SDDE \eqref{eq:SDDEind}, and provide a rigorous definition of the strong convergence of a numerical method to the true solution of \eqref{eq:SDDEind} via a continuous-time interpolant of the scheme. In Section \ref{sec:moments} we derive moment bounds on important processes related to the SDDE \eqref{eq:SDDEind} and its discretisation. In Section \ref{sec:main} we present our main strong convergence result for the Lie-Trotter splitting method. 

In Section \ref{sec:numerics} we numerically verify the error bounds derived in Section \ref{sec:main} for test equations with oscillatory and nonoscillatory trajectories, and observe the relationship between the convergence bounds and the strength of the correlation coefficient $\rho$. To aid readability, a presentation of the proofs of our main results is deferred to Section \ref{sec:proofs}.

\section{Existence and uniqueness of SDDE solutions}\label{sec:exact}
Let $(\Omega,\mathcal{F},(\mathcal{F}_t)_{t\in[-\tau,T]},\mathbb{P})$ be a complete probability space equipped with filtration $(\mathcal{F}_t)_{t\in[-\tau,T]}$. In \eqref{eq:SDDEcorr} and \eqref{eq:SDDEind} respectively we suppose that $W$ and $B$ are $\mathcal{F}_t$-adapted and independent of $\mathcal{F}_0$.  We also suppose that the initial data $\psi$ is an $\mathcal{F}_0$-measurable random function, right continuous in $\mathbb{R}$ and $\mathbb{E}\left[\|\psi\|_\infty^p\right]<\infty$ for any $p\geq 2$, where $\mathbb{E}$ is the expectation with respect to $\mathbb{P}$. Here we define the norm $\|\psi\|_\infty:=\sup_{-\tau\leq t\leq 0}|\psi(t)|$.

Under the conditions of the following assumption, \eqref{eq:SDDEind} has a unique strong solution over the interval $[-\tau,T]$ for any $T<\infty$, adapted to the filtration $(\mathcal{F}_{t})_{t\in[-\tau,T]}$. 
\begin{assumption}\label{assum:LGcond}
Suppose that $f$  and each $g_i$ satisfy a linear growth condition: 
 there exists a positive constant $K_T$ such that for all $x\in\mathbb{R}$
\begin{equation}\label{eq:linear-growth}
\max\left\{|f(x)|^2,|g(x)|^2\right\}\leq K_T(1+|x|^2).
\end{equation}
\end{assumption}
Note that $f,g$ are pure delay functions and take no input from the present state $X(t)$. Therefore local Lipschitz conditions are not required to ensure existence and uniqueness; see for example \cite[Chapter 5.3]{Mao2008}. 

Proceed by the method of steps. First consider the inhomogeneous scalar linear SDE
\begin{equation}\label{eq:genSDE}
dx(t)=\left[\mu x(t)+f(t)\right]dt+\sum_{i=1}^{2}\left[b_ix(t)+g_i(t)\right]dB_i(t),\quad t\in[t_0,T],
\end{equation}
where the initial value $x(t_0)=x_0$ is an $\mathcal{F}_{t_0}$-measurable and $L^p(\Omega,\mathbb{R})$-valued random variable for some $p\geq 2$, where $f$ and each $g_i$ are Borel measurable and bounded $\mathbb{R}$-valued functions on $[t_0,T]$. The unique solution of \eqref{eq:genSDE} can be expressed for any $t\in[t_0,T]$ in variation of constants form as
\begin{equation}\label{eq:VoCform}
x(t)=\Phi_{t_0}(t)\left(x_0+\int_{t_0}^{t}\Phi_{t_0}^{-1}(s)\left[f(s)-\sum_{i=1}^{2}b_ig_i(s)\right]ds+\sum_{i=1}^{2}\int_{t_0}^{t}\Phi_{t_0}^{-1}(s)g_i(s)dB_i(s)\right),
\end{equation}
where $\Phi_{t_0}(t)$ is the fundamental solution of the corresponding homogeneous equation, given by
\begin{equation}\label{eq:phi}
\Phi_{t_0}(t)=\exp\left\{\left(\mu-\frac{1}{2}\sum_{i=1}^{2}b_i^2\right)(t-t_0)+\sum_{i=1}^{2}b_i (B_i(t)-B_i(t_0))\right\}.
\end{equation}
See, for example, Mao~\cite[Chapter 3]{Mao2008}. Note that $\Phi_s(t)$ defined in \eqref{eq:phi} admits the decomposition 
\begin{equation}\label{eq:PhiDecomp}
\Phi_s(t)=\Phi_s(u)\Phi_u(t),\quad 0\leq s\leq u\leq t\leq T,
\end{equation}
and moreover, $\Phi_s(t)$ is independent of $\mathcal{F}_s$. 

This gives us an avenue to construct a solution for \eqref{eq:SDDEind} successively over each interval $((j-1)\tau,j\tau]$, $j=0,\ldots,M$, where $M=T/\tau$ (suppose without loss of generality that $M\in\mathbb{N}$). Set 
\[
b_1=\sqrt{1-\rho^2}\sigma;\quad b_2=\rho\sigma;\quad g_1\equiv 0;\quad g_2\equiv g.
\]

Denoting on each interval $(j\tau,(j+1)\tau]$ the $\mathcal{F}_{j\tau}$-measurable process $\psi_j(t):=X(t-\tau)$, an application of the SDE~\eqref{eq:genSDE} over the interval $(j\tau,(j+1)\tau]$ with $t_0=j\tau$ gives, for $j=0,\ldots,M-1$,
\begin{multline}
X(t)=\Phi_{j\tau}(t)\left(X(j\tau)+\int_{j\tau}^{t}\Phi_{j\tau}^{-1}(s)\left[f(\psi_j(s))-\rho\,\sigma g(\psi_j(s))\right]ds\right.\\
\left.+\int_{j\tau}^{t}\Phi_{j\tau}^{-1}(s)g(\psi_j(s))dB_2(s)\right),\quad t\in(j\tau,(j+1)\tau],\label{eq:varConstGen}
\end{multline}
where, specifying \eqref{eq:phi} for the SDDE \eqref{eq:SDDEind},
\begin{equation*}
\Phi_{s}(t)=\exp\left\{\left(\mu-\frac{1}{2}\sigma^2\right)(t-s)+\sqrt{1-\rho^2}\,\sigma (B_1(t)-B_1(s))+\rho\,\sigma(B_2(t)-B_2(s))\right\}.
\end{equation*}
We can see that when $j=0$, $\psi_0(t)=\psi(t)$ corresponds to the $\mathcal{F}_0$-measurable initial data, allowing us to use \eqref{eq:varConstGen} to uniquely construct $\psi_1(t)$ over the first step $t\in(0,\tau]$. We can then construct stepwise $\psi_2(t)$ over $t\in(\tau,2\tau]$, $\psi_3(t)$ over $t\in(2\tau,3\tau]$, and so on, until the solution on the full interval $[-\tau,T]$ has been constructed in its entirety.

\section{Splitting Schemes for SDDEs}\label{sec:split}
In this section, we define the Lie-Trotter and Strang splitting schemes, and provide a precise characterisation of the strong convergence of a numerical approximation to the true solution of \eqref{eq:SDDEind}. Fix $N\in\mathbb{N}$ and let $\{0=t_0,t_1,\ldots,t_N=T\}$ be a uniform mesh over the interval $[0,T]$, so that each $t_n=n\delta t$, where $\delta t=T/N<\min\{\tau,1\}$. We allow this notation to support half-steps, whereby $t_{n+\frac{1}{2}}=(n+1/2)\delta t$. 

\subsection{Lie-Trotter and Strang splitting strategies}

Decompose \eqref{eq:SDDEind} into the following subequations:
\begin{equation}\label{eq:y1}
\begin{aligned}
dY^{[1]}(t)
&=
f\left(Y^{[1]}(t-\tau)\right)\,dt
+
g\left(Y^{[1]}(t-\tau)\right)\,dB_2(t),
\qquad t\in(0,T],\\
Y^{[1]}(t)
&=
Y^{[1]}_{[-\tau,0]}=\psi(t),
\qquad t\in[-\tau,0].
\end{aligned}
\end{equation}
and
\begin{equation}\label{eq:y2}
\begin{aligned}
dY^{[2]}(t) &= \mu Y^{[2]}(t)\,dt + \sqrt{1-\rho^2} \sigma Y^{[2]}(t)\,dB_1(t)+\rho\sigma Y^{[2]}(t)\,dB_2(t), \qquad t\in(0,T],\\
Y^{[2]}(0) &= Y^{[2]}_0 = 1.
\end{aligned}
\end{equation}

A splitting strategy proposes to solve each of subequations \eqref{eq:y1} and \eqref{eq:y2} separately and sequentially over each timestep, using the output of one as an initial condition for the other. 

The first subequation may be written in integral form as
\begin{multline}\label{eq:y1step}
Y^{[1]}\left((n+1)\delta t\right)
=
Y^{[1]}(n\delta t)
\\+
\int_{n\delta t}^{(n+1)\delta t}
f\left(Y^{[1]}(s-\tau)\right)\,ds
+
\int_{n\delta t}^{(n+1)\delta t}
g\left(Y^{[1]}(s-\tau)\right)\,dB_2(s),
\end{multline}
and this may be solved over each step by a suitable approximation of the integrals on the right-hand side. For example if we use the Euler-Maruyama method to solve \eqref{eq:y1step} we have 
\[
Y^{[1]}_{n+1}=\Psi^{[1]}(t_{n+1};t_n,Y^{[1]}_{n})
\]
where 
\[
\Psi^{[1]}(t;s,r,r_\tau)=r+f(r_\tau)|t-s|+g(r_\tau)(B_2(t)-B_2(s)),
\]
for $t_n\leq s<t\leq t_{n+1}$ and any $n=0,\ldots,N-1$. The second subequation may be solved exactly as
\begin{equation*}
Y^{[2]}_{n+1}=\Psi^{[2]}\left(t_{n+1};t_n,Y^{[1]}_{n+1}\right)
\end{equation*}
where
\begin{multline}\label{eq:Psi2}
\Psi^{[2]}(t;s,r)
\\=r\exp\left\{\left(\mu-\frac{1}{2}\sigma^2\right)|t-s|+\sqrt{1-\rho^2}\sigma(B_1(t)-B_1(s))+\rho\sigma (B_2(t)-B_2(s))\right\},
\end{multline}
again for $t_n\leq s<t\leq t_{n+1}$ and any $n=0,\ldots,N-1$.
Given the maps $\Psi^{[1]}$ and $\Psi^{[2]}$, we can now characterise various splitting strategies according to the manner in which these maps are composed. Two common examples are as follows.

\begin{enumerate}
\item The Lie--Trotter splitting scheme is given over each step of length $\delta t$ by the composition
\begin{equation}\label{eq:LTmap}
X_{n+1}=\Psi^{[2]}\left(t_{n+1};t_n,\Psi^{[1]}\left(t_{n+1};t_n,X_n,X_{n-\lfloor \tau/{\delta t}\rfloor}\right)\right).
\end{equation}

\item The Strang splitting scheme is characterised over each step of length $\delta t$ by the composition
\begin{equation}\label{eq:Strmap}
X_{n+1}=\Psi^{[1]}\left(t_{n+1};t_{n+\frac{1}{2}},\Psi^{[2]}\left(t_{n+1};t_n,\Psi^{[1]}\left(t_{n+\frac{1}{2}};t_n,X_n,X_{n-\lfloor \tau/{\delta t}\rfloor}\right)\right),X_{n-\lfloor \tau/{\delta t}\rfloor}\right).
\end{equation}
\end{enumerate}

\subsection{Continuous interpolants of splitting schemes}
The Lie-Trotter scheme \eqref{eq:LTmap}, where solution map $\Psi^{[1]}$ is an approximation and based upon the Euler-Maruyama method, may be written in stochastic difference equation form as
\begin{equation}\label{eq:LT-multi}
X_{n+1}=\Phi_{t_n}(t_{n+1})\left(
X_n+\delta t\, f\left(X_{n-\lfloor\tau/\delta t\rfloor}\right)+g\left(X_{n-\lfloor\tau/\delta t\rfloor}\right)\left(B_2(t_{n+1})-B_2(t_{n})\right)\right),
\end{equation}
for $n=0,\ldots,N-1$. This is appropriate for implementation in code. However, for our strong convergence analysis we will work with a continuous-time interpolant of \eqref{eq:LT-multi}, given over the step $[t_n,t_{n+1}]$ by
\begin{equation}\label{eq:lieinter}
\begin{aligned}
Y(t)&=\Phi_{t_n}(t)\left(Y(t_n)+\int_{t_n}^{t}f\left(Y\left(t_{n-\lfloor\tau/\delta t\rfloor}\right)\right)\,ds
+\int_{t_n}^{t}g\left(Y\left(t_{n-\lfloor\tau/\delta t\rfloor}\right)\right)\,dB_2(s)\right).
\end{aligned}
\end{equation}
Notice that the discrete and continuous variants of the scheme coincide at the meshpoints, so that $Y(t_n)=X_n$ for all $n=0,\ldots,N$.

The Strang scheme can be written as
\begin{align}
X_{n+1}=&\left(X_n + \frac{\delta t}{2}f\left(X_{n-\lfloor\tau/\delta t\rfloor}\right)+g\left(X_{n-\lfloor\tau/\delta t\rfloor}\right)\,\left(B_2\left(t_{n+\frac{1}{2}}\right)-B_2(t_n)\right)\right)\Phi_{t_n}(t_{n+1})\nonumber
\\&+\frac{\delta t}{2}f\left(X_{n-\lfloor\tau/\delta t\rfloor}\right)+g\left(X_{n-\lfloor\tau/\delta t\rfloor}\right)\,\left(B_2\left(t_{n+1}\right)-B_2\left(t_{n+\frac{1}{2}}\right)\right),\label{eq:Strang-discrete}
\end{align}
for $n=0,\ldots,N-1$. We do not present the continuous interpolant of \eqref{eq:Strang-discrete}, since the Strang scheme is investigated only numerically in this article, and \eqref{eq:Strang-discrete} is sufficient for this. Nonetheless, we will define in the next subsection the notion of strong convergence of a numerical method in terms of its continuous interpolant.

\subsection{Strong convergence of splitting schemes}

Suppose that $Y$ is the solution of the continuous interpolant of either a Lie-Trotter or a Strang splitting as described in the previous subsection, and let $X$ be the solution of \eqref{eq:SDDEind}. We say that $Y$ converges strongly to $X$ in $p^{th}$-moment with order $\gamma$ if and only if there exists $C>0$ independent of $\delta t$ such that
\begin{equation}\label{eq:strongconver}
\max_{t\in[0,T]}\mathbb{E}\left[|X(t)-Y(t)|^p\right]\leq C\delta t^{\gamma p}.
\end{equation}
If $p=2$ we say that $Y$ converges to $X$ in mean-square with order $\gamma$, and we are concerned with this case here.

\section{Moment bounds on key processes}\label{sec:moments}
\subsection{Moment bounds on the fundamental solution}
In our analysis we will require the following bounds which are special cases of moment bounds applying to matrix-valued variants of $\Phi_s(t)$ provided in Erdogan \& Lord~\cite{ErdoganLord2022}. In what follows, let $L_p(\Omega,\mathbb{R})$ denote the space of scalar real-valued random variables $U$ on the probability space $(\Omega,\mathcal{F},\mathbb{P})$ that satisfy $\mathbb{E}\left[|U|^p\right]<\infty$.
\begin{lemma}\label{lem:2.1}
Let $0\leq s<t\leq T$, $0<t-s<1$, $p\ge 2$ and consider the process $\Phi_s(t)$ given in \eqref{eq:phi}. For each $\mathcal{F}_s$-measurable random variable $U\in L_p(\Omega,\mathbb{R})$, there exists a constant $C_p$ such that
\begin{equation}
\mathbb{E}\left[\,|U-\Phi_s(t)U|^p\,\right] \le C_p\,|t-s|^{p/2}.\label{eq:Phi1bound}
\end{equation}
\end{lemma}

\begin{lemma}
\label{lem:2.2}
Suppose $p\ge 2$, $0\leq s<t\leq T$ and consider the process $\Phi_s(t)$ given in \eqref{eq:phi}. Then
\begin{enumerate}
\item for any $\mathcal{F}_s$-measurable random variable $V\in L_p(\Omega,\mathbb{R})$,
\begin{equation}\label{eq:Phi-bound-v}
\mathbb{E}\left[|\Phi_s(t)V|^{p}\right]
\;\le\;
\exp\left[\left(p|\mu|+\frac{p(p-1)\sigma^2}{2}\right)(t-s)\right]
\,\mathbb{E}\left[|V|^{p}\right].
\end{equation}

\item for any $\mathcal{F}_t$-measurable random variable $U\in L_p(\Omega,\mathbb{R})$,
\begin{equation}\label{eq:Phi-bound-u}
\mathbb{E}\left[|\Phi_s(t)U|^{p}\right]\;\le\; K_r^{p/r}\,\mathbb{E}\left[|U|^{q}\right]^{p/q},
\end{equation}
where $K_{r}=\mathbb{E}\left[|\Phi_s(t)|^{r}\right]<\infty$ for $1/p=1/r+1/q$ with $r,q>1$.
\end{enumerate}
\end{lemma}

\subsection{Moment properties of the Lie-Trotter scheme}
For our main convergence analysis, we require the following mean-square bound on the continuous interpolant of the Lie-Trotter splitting given by the SDDE \eqref{eq:lieinter}:
\begin{lemma}\label{lem:2.3new}
Let $Y$ be the continuous interpolant of the Lie-Trotter splitting scheme given by the SDDE \eqref{eq:lieinter}, and assume that the linear growth condition \eqref{eq:linear-growth} in Assumption \ref{assum:LGcond} holds on $f$ and $g$. Then, for any $p\geq 2$ there exists a constant $C>0$ such that for all $t\in[0,T]$
\begin{equation}\label{eq:lemma23new}
\mathbb{E}\left[\sup_{s\in[-\tau,t]}|Y(s)|^p\right]\leq C.    
\end{equation}
\end{lemma}

Next, we state a result on the path regularity of the scheme.
\begin{lemma}\label{lem:4.4}
Let $Y$ be the continuous interpolant of the Lie-Trotter splitting scheme given by the SDDE \eqref{eq:lieinter}, and assume that the linear growth condition \eqref{eq:linear-growth} in Assumption \ref{assum:LGcond} holds on $f$ and $g$.  Suppose $p\ge 1$ and fix a step $[t_n,t_{n+1}]$. For all $t_n\le s<t\le t_{n+1}$, there is a constant $C>0$ such that
\begin{equation*}
 \mathbb{E}\left[|Y(t)-Y(s)|^{p}\right] \;\le\; C|t-s|^{p/2}.   
\end{equation*}
\end{lemma}
The proofs of Lemmas \ref{lem:2.3new} and \ref{lem:4.4} are deferred to Section \ref{sec:proofs}.

\section{Main theoretical result}\label{sec:main}

In this section we will provide a bound on the strong error of the Lie-Trotter splitting method, where the nonlinear subsystem with delay has been discretised using the Euler-Maruyama method. This bound confirms order-$1/2$ strong convergence in mean-square when $\rho=0$, and provides a uniform nonzero bound in the case where $\rho\neq 0$. The proof relies upon the following  Gronwall inequality for delay-type equations.

\begin{lemma}[Delay Gronwall Inequality~\cite{GyorHorvath2016}]\label{le:delaygrownwall}
Let $t_0\in\mathbb{R}$, $t_0<T\le\infty$, and $c\ge 0$. Let $a,b:[t_0,T]\to\mathbb{R}_+$ be locally integrable. 
Assume $r\ge 0$ and $\tau:[t_0,T]\to\mathbb{R}_+$ is such that
\[
t_0-r \le t-\tau(t),\qquad t_0\le t<T.
\]
Set $b(t)=0$ for $t_0-r\le t<t_0$.
If a Borel measurable, locally bounded function $x:[t_0-r,T]\to\mathbb{R}_+$ satisfies
\begin{equation*}
x(t)\le c + \int_{t_0}^{t} b(u)\,x(u)\,du + \int_{t_0}^{t} a(u) x\left(u-\tau(u)\right)du,
\qquad t_0\le t<T,
\end{equation*}
then
\begin{equation*}
x(t)\le K \exp\left(
\int_{t_0}^{t} b(s)\,ds
+\int_{t_0}^{t} a(s)\exp\left(-\int_{s-\tau(s)}^{s} b(u)\,du\right)\,ds
\right),\qquad t_0\le t<T,
\end{equation*}
where
\[
K := \max\left( c, \sup_{\,t_0-r \le s \le t_0} x(s) \right).
\]
\end{lemma}
We also require a technical condition on the drift coefficient of the SDDE governed by \eqref{eq:SDDEind} beyond what is required for the existence and uniqueness of solutions.
\begin{assumption}\label{ass:delayOSL}
Let $f,g$ be the nonlinear coefficients of \eqref{eq:SDDEind}. Assume  that for all $x,y\in\mathbb{R}$ there exists a positive constant $C>0$ with
\begin{equation}\label{eq:4.5}
\langle x-y,\; f(x-\tau)-f(y-\tau)\rangle \le C\,|x-y|^{2},    
\end{equation}
and
\begin{equation}\label{eq:gLipschitz}
    |g(x)-g(y)|\leq C|x-y|.
\end{equation}
\end{assumption}
Note that our analysis also works if condition \eqref{eq:4.5} in Assumption \ref{ass:delayOSL} is replaced with a global Lipschitz bound on $f$. Moreover, global Lipschitz bounds on $f$ and $g$ are sufficient to ensure the linear growth bounds captured in \eqref{eq:linear-growth} in Assumption \ref{assum:LGcond}.

We now present our main convergence result.
\begin{theorem}\label{thm:main}
Let $X$ be a solution of \eqref{eq:SDDEind}, and let $Y$ be the continuous interpolant of the Lie-Trotter splitting method for \eqref{eq:SDDEind}, given by \eqref{eq:lieinter}. 
Then, for all $\rho\in(-1,1)$ there exist positive constants $K_1,K_2>0$, independent of $\delta t$ and $\rho$, such that
\begin{equation}\label{eq:mainBound}
    \max_{t\in[0,T]}\mathbb{E}\left[|X(t)-Y(t)|^2\right]\leq K_1\delta t+K_2|\rho|.
\end{equation}
In the special case that $\rho=0$, the Lie-Trotter splitting method is strongly mean-square-convergent with order $1/2$.
\end{theorem}
The proof of Theorem \ref{thm:main} relies directly upon Lemmas \ref{lem:2.1}--\ref{lem:4.4}, and \ref{le:delaygrownwall}. A detailed presentation is deferred to Section \ref{sec:proofs}. 

Notice from \eqref{eq:mainBound} in the statement of Theorem \ref{thm:main} that if $\rho\neq 0$ the global strong error is bounded above uniformly in $\delta t$, and we can neither confirm nor rule out strong convergence. We investigate this numerically in the next section.

\section{Numerical Experiments}\label{sec:numerics}
In this section, we present numerical experiments to explore the theoretical result provided by Theorem \ref{thm:main} in Section~\ref{sec:main}. We will do this for both the Lie-Trotter and Strang splitting schemes for two example SDDEs.

\subsection{Numerical estimation of the strong convergence order $\gamma$.}
Motivated by the definition of
strong convergence in \eqref{eq:strongconver}, we assume that the root-mean-square error satisfies for some $C>0$
\begin{equation}\label{eq:numErrorRate}
\mathcal{E}_{\text{RMS}}:=\sqrt{\max_{k=0,\ldots,N}\mathbb{E}\left[|X(t_k)-Y(t_k)|^2\right]}
\approx C (\delta t)^{\gamma}.
\end{equation}
We then take logarithms on both sides of \eqref{eq:numErrorRate}:
\[
\log \mathcal{E}_{\text{RMS}}\approx \log(C)+\gamma\log(\delta t).
\]
In the numerical experiments, the expectation in $\mathcal{E}_{\text{RMS}}$ is approximated simulating the pathwise error between the reference solution and numerical approximation over a sample of $M'=500$ simulated trajectories. As reference solution, we use the variation of constants form of the exact solution of the SDDE given by~\eqref{eq:VoCform}, where over a very fine mesh of size $2^{-18}$, the integral for the drift has been approximated using the trapezoid rule, and the integral for the drift has been approximated using an explicit rectangular approximation. The value of $\gamma$ is then estimated as the slope of the regression line of $\log(\mathcal{E}_{\text{RMS}})$ on $\log(\delta t)$.

\subsection{Investigating the effect of changing $\rho$ on numerical convergence order.}
We will numerically investigate the strong convergence order of the Lie-Trotter and Stang splitting schemes given by \eqref{eq:LT-multi} and \eqref{eq:Strang-discrete} respectively. 

\begin{exmp}\label{ex:one} 
First consider consider the SDDE \eqref{eq:SDDEind} with
\[
f(x) = -x, \qquad g(x) = x,
\]
and parameters
\[
\mu = 0,\qquad \sigma = 1.2,\qquad \tau = 1,\qquad T = 8,\qquad \rho \in [-1,1].
\]
The initial condition is $X(t)=1$ for $t\in[-\tau,0]$.
The reference solution is computed with time step $2^{-18}$, whereas the numerical
schemes are tested with
\begin{equation}\label{eq:stepsizes}
\delta t = \{\,2^{-10},\, 2^{-11},\, 2^{-12},\, 2^{-13},\, 2^{-14}\}.
\end{equation}

Figure~\ref{F:1} illustrates the observed numerical order of strong convergence with respect to the correlation parameter $\rho$ for both the Lie--Trotter and Strang splitting schemes. We see that the numerical order of strong convergence is consistent with $\gamma=1/2$ when $\rho=0$, and declines rapidly but smoothly to zero as the size of $|\rho|$ increases. This result is consistent with the statement of Theorem \ref{thm:main} in the case of the Lie-Trotter scheme, and shows that the Strang scheme behaves similarly for the same numerical example.
\end{exmp}

\begin{figure}
\begin{center}
 \includegraphics[width=0.9\textwidth]{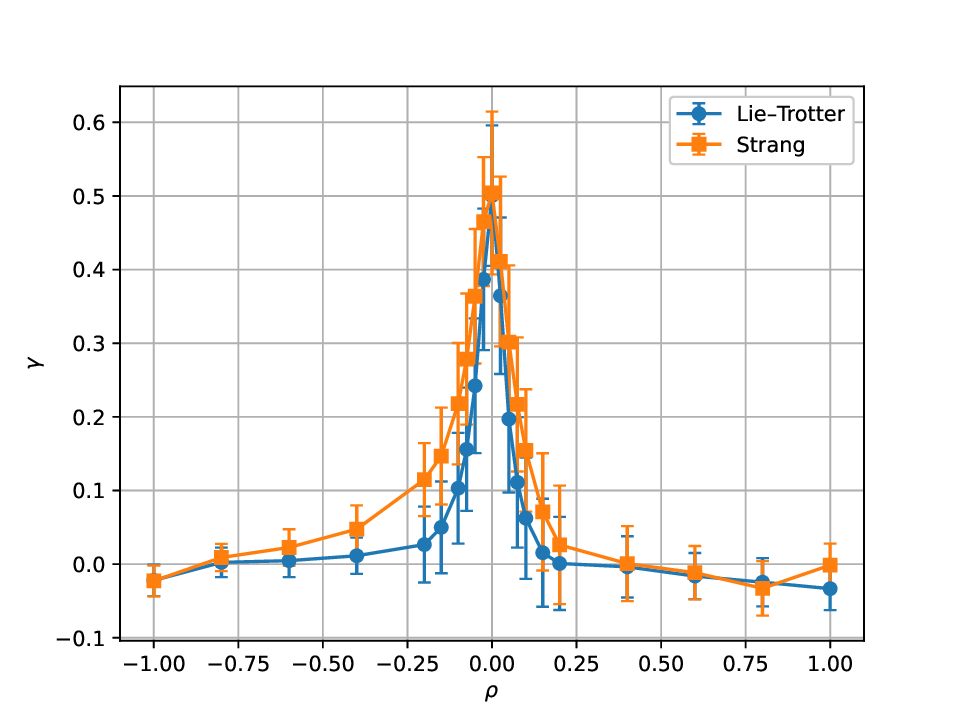}
\end{center}
\caption{Strong convergence order $\gamma$ against $\rho$ (Example~\ref{ex:one}).}\label{F:1}
\end{figure}

\begin{exmp}\label{ex:two} 
For comparison, we consider the case of \eqref{eq:SDDEind} with
\[
f(x) = 0.6x, \qquad g(x) = x,
\]
and
\[
\mu = -0.4,\qquad \sigma = 1.2,\qquad \tau = 1,\qquad T = 8,\qquad \rho \in [-1,1].
\]
We use the same initial condition, reference time step and $\delta t$ values as in Example \ref{ex:one}.

Figure~\ref{F:2} presents a plot of numerical order of strong convergence against $\rho$. In particular, when $\rho = 0$, the numerical convergence order is close to $\gamma=1/2$, in agreement with the theoretical results established in Subsection~\ref{sec:main}, and we observe that the numerical order of strong convergences drops to zero as $|\rho|$ increases.
\end{exmp}

\begin{figure}
\begin{center}
\includegraphics[width=0.9\textwidth]{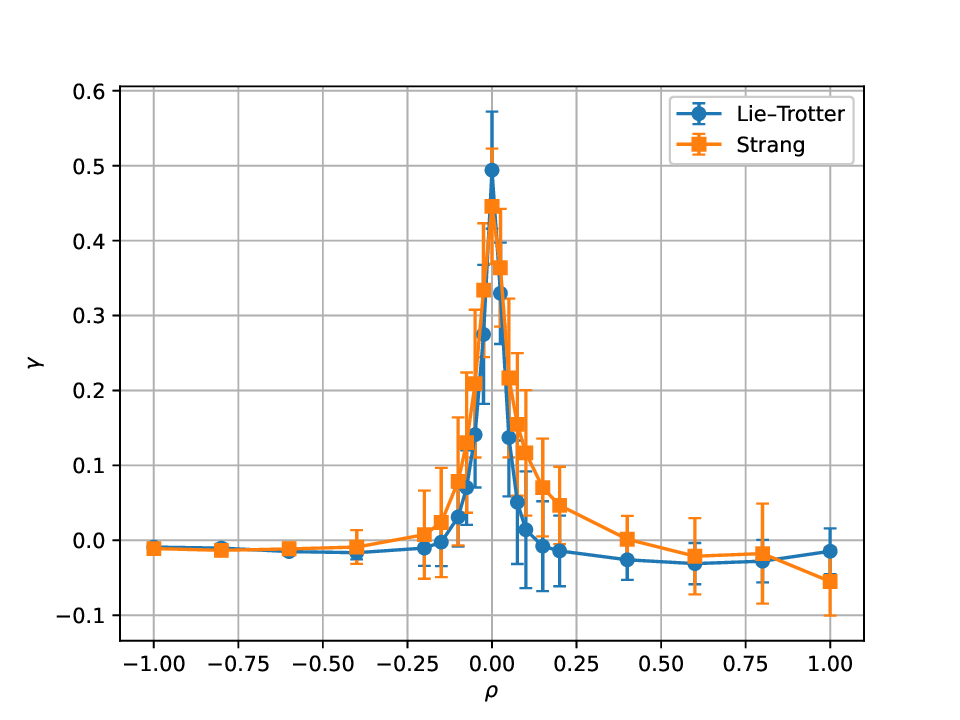}
\end{center}
\caption{Strong convergence order $\gamma$ against $\rho$ (Example~\ref{ex:two}).}\label{F:2}
\end{figure}

We finish by noting that, in both Figure \ref{F:1} and \ref{F:2}, we take 20 groups of 25 trajectories to estimate the standard deviation in the error bars. Once the error bar associated with an observed order of convergence includes zero, we may conclude that the method is not converging numerically over the stepsize range \eqref{eq:stepsizes} for this value of $\rho$.

\subsection{An alternative method: varying the split}

The proof of Theorem \ref{thm:main} in Section \ref{sec:proofs} relies on an application of the stochastic product rule to solutions (exact or approximate) of \ref{eq:y1} and \ref{eq:y2}. Since these subsystems share the standard Brownian motion $B_2$ in common, this creates a square variation term that contributes to the overall error bound in the analysis. This term can be eliminated by moving the diffusion term with $B_2$ in \ref{eq:y2} to \eqref{eq:y1}, suggesting an alternative splitting based upon the new decomposed SDE subsystems
\begin{align*}
dY^{[1]}(t) &= f(Y^{[1]}(t-\tau))\,dt 
+ \left[g(Y^{[1]}(t-\tau))+ \rho\sigma Y^{[1]}(t)\right]dB_2(t), \quad t \in (0,T];\\
Y^{[1]}(t)&=Y^{[1]}_{[-\tau,0]}=\psi(t),\quad t\in[-\tau,0],
\end{align*}
and
\begin{align*}
dY^{[2]}(t) &= \mu Y^{[2]}(t)\,dt 
+ \sqrt{1-\rho^2}\,\sigma Y^{[2]}(t)\,dB_1(t) 
, \quad t \in (0,T];\\
Y^{[2]}(0)&=Y^{[2]}_0=1.
\end{align*}
with the construction and composition of associated maps as described according to the Lie-Trotter approach in Section \ref{sec:split}. For this new variant, we have confirmed that an error bound of the form \eqref{eq:mainBound} in the statement of Theorem \ref{thm:main} still holds; the proof is similar to that of Theorem \ref{thm:main} and omitted for brevity. As such, we do not see an immediate theoretical advantage to this variant. However, in practice, we observe in Figures \ref{F:3}-\ref{F:4} that the range of values of $\rho$ over which nonzero order of strong convergence is observed has increased, improving the useful domain of the method.

\begin{figure}
\begin{center}
 \includegraphics[width=0.9\textwidth]{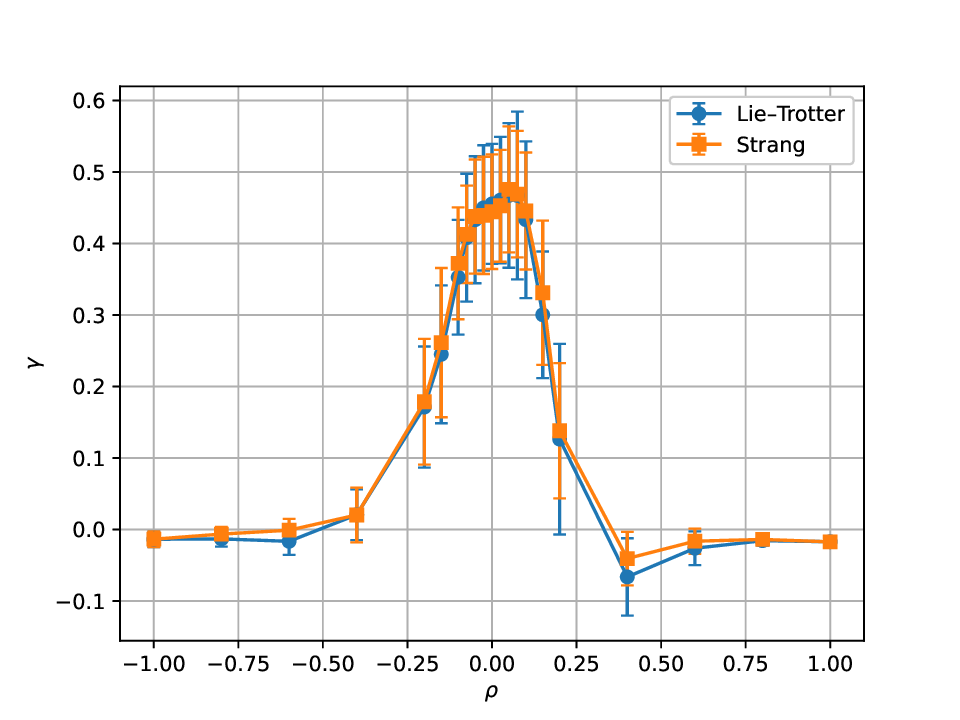}
\end{center}
\caption{Strong convergence for Lie-Trotter variant showing order $\gamma$ against $\rho$ (Example~\ref{ex:one}).}\label{F:3}
\end{figure}

\begin{figure}
\begin{center}
 \includegraphics[width=0.9\textwidth]{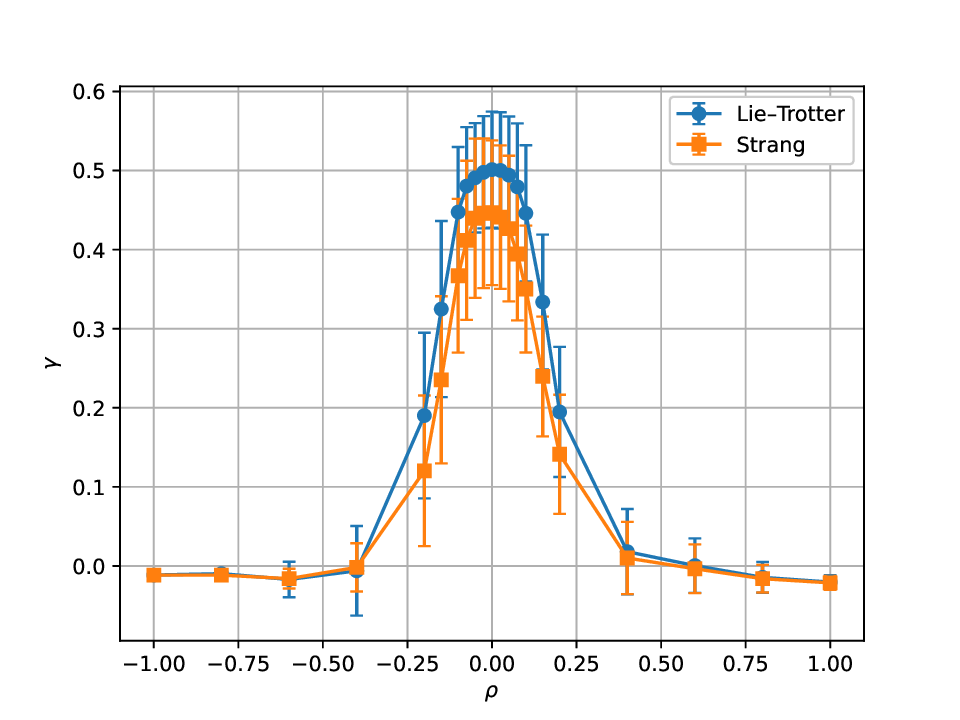}
\end{center}
\caption{Strong convergence for Lie-Trotter variant showing order $\gamma$ against $\rho$ (Example~\ref{ex:two}).}\label{F:4}
\end{figure}

\section{Conclusions and further work}
We have shown that the Lie-Trotter splitting scheme is strongly convergent of order $1/2$ for a class of semilinear SDDEs when the noise perturbations on the delay and non-delay parts are independent. In the case where the SDDE contains correlated noises, we can ensure an upper bound on the global error of the Lie-Trotter scheme, however this bound is independent of the stepsize $\delta t$ and does not guarantee convergence. Numerically we see that the observed convergence order reduces to zero rapidly as $|\rho|$ increases, and similar behaviour is seen for Strang splitting. By selecting a different decomposition of the SDE into subsystems that avoids repeating instances of the same standard Brownian motions across different subsystems, we can improve the range of $\rho$ over which numerical strong convergence is observed, though the form of the obtainable theoretical error bound is the same. 

We conclude that these specific splitting strategies, though intuitively natural, are only suitable for the numerical solution of such SDDEs in the case where the perturbing noises are independent or (at best) very weakly correlated. A variant that ensures the independence of noise perturbations across subsystems is seen to be more numerically effective. 

There are many possible future directions for this work. These include developing a lower bound on the order of strong convergence, and extending the theoretical results in Sections \ref{sec:moments} and \ref{sec:main} to finite-dimensional systems and to splitting schemes beyond Lie-Trotter; identifying further splitting strategies that are numerically effective when $\rho\neq 0$; investigating the dynamic consistency of these schemes, including their ability to capture oscillatory behaviour, and to empirically reproduce stationary distributions.

\section{Proofs}\label{sec:proofs}
\begin{proof}[Proof of Lemma \ref{lem:2.3new}]
Fix any $n=0,\ldots,N-1$. 
For $t\in[t_n,t_{n+1}]$ the solution of \eqref{eq:lieinter} may be written $Y(t)=Z_1(t)Z_2(t)$, where 
\begin{align*}
dZ_1(t)&=\mu Z_1(t)dt+\sqrt{1-\rho^2}\sigma Z_1(t)dB_1(t)+\rho\sigma Z_1(t)dB_2(t);\\
dZ_2(t)&=f(\bar Y(t_n))dt+g(\bar Y(t_n))dB_2(t),
\end{align*}
with initial values $Z_1(t)=1$, $Z_2(t_n)=Y(t_n)$, and using the notation $\bar Y_n:=Y(t_{n-\lfloor\tau/\delta t\rfloor})$. An application of the stochastic product rule followed by the transformation $V(y)=|y|^p$ leads to the semimartingale representation  
\begin{multline*}
|Y(t)|^p=|Y(t_n)|^p
\\+p\int_{t_n}^{t}\left[\Phi_{t_n}(s)f(\bar Y_n)|Y(s)|^{p-1}+\rho\sigma\Phi_{t_n}(s)g(\bar Y_n)|Y(s)|^{p-1}+\mu|Y(s)|^p\right]ds
\\+\frac{p(p-1)}{2}\int_{t_n}^{t}\left[\sigma^2|Y(s)|^p+\Phi_{t_n}(s)^2g(\bar Y_n)^2|Y(s)|^{p-2}+2\rho\sigma\Phi_{t_n}(s)g(\bar Y_n)|Y(s)|^{p-1}\right]ds
\\+p\int_{t_n}^{t}\sqrt{1-\rho^2}\sigma|Y(s)|^pdB_1(s)+p\int_{t_n}^{t}\left[\rho\sigma|Y(s)|^p+\Phi_{t_n}(s)g(\bar Y_n)|Y(s)|^{p-1}\right]dB_2(s),
\end{multline*} 
for $t\in[t_n,t_{n+1}]$. For brevity, we outline the proof under the assumption that both $f$ and $g$ are globally Lipschitz continuous and $f(0)=g(0)=0$, in which case there exists $C>0$ such that 
\begin{equation}\label{eq:FGtightbounds}
|f(y)|\vee |g(y)|\leq C|y|\text{ for all }y\in\mathbb{R}.
\end{equation} 
A proof that requires only the linear growth bound \eqref{eq:linear-growth} on $f$ and $g$ is more complex, and may be achieved using the alternative transformation $V(y)=(1+|y|^2)^{p/2}$ (see for example Theorem 5.4.1 in Mao~\cite{Mao2008}).

Iterate both sides back to the initial data and make use of \eqref{eq:FGtightbounds} and the fact that $|Y(0)|\leq \|\psi\|$ to get 
\begin{multline*}
|Y(t)|^p\leq \|\psi\|^p+C\int_{0}^{t}\left[|Y(r)|^p+|Y(r)|^{p-1}\sum_{j=0}^{N-1}\left(|\bar Y_j|\Phi_{t_j}(r)\mathcal{I}_{\{r\in[t_j,t_{j+1}]\}}\right)\right.
\\ \left.\qquad\qquad\qquad\qquad\qquad\qquad+|Y(r)|^{p-2}\sum_{j=0}^{N-1}\left(|\bar Y_j|^2\Phi_{t_j}(r)^2\mathcal{I}_{\{r\in[t_j,t_{j+1}]\}}\right)\right]dr
\\+C\int_{0}^{t}|Y(r)|^pdB_1(r)+C\int_{0}^{t}|Y(r)|^pdB_2(r)
\\+C\int_{0}^t\sum_{j=0}^{N-1}\left(|\bar Y_j|\Phi_{t_j}(r)\mathcal{I}_{\{r\in[t_j,t_{j+1}]\}}\right)|Y(r)|^{p-1}dB_2(r),
\end{multline*}
where $C>0$ is a generic constant and $\mathcal{I}_A$ is the indicator of the set $A$. Take the supremum over the time set on both sides
\begin{multline*}
\sup_{s\in[-\tau,t]}|Y(s)|^p\leq \|\psi\|^p
\\+C\int_{0}^{t}\left[\sup_{u\in[-\tau,s]}|Y(u)|^p\left(1+\sum_{j=0}^{N-1}\left(\Phi_{t_j}(s)+\Phi_{t_j}(s)^2\right)\mathcal{I}_{\{s\in[t_j,t_{j+1}]\}}\right)\right]ds
\\+C\sup_{s\in[0,t]}\int_{0}^{s}|Y(r)|^pdB_1(r)+C\sup_{s\in[0,t]}\int_{0}^{s}|Y(r)|^pdB_2(r)
\\+C\sup_{s\in[0,t]}\int_{0}^s\sum_{j=0}^{N-1}\left(|\bar Y_j|\Phi_{t_j}(r)\mathcal{I}_{\{r\in[t_j,t_{j+1}]\}}\right)|Y(r)|^{p-1}dB_2(r)
\end{multline*}
We may now take expectations on both sides, invoke the Burkholder-Davis-Gundy inequality to control the diffusion terms, and apply the bounds \eqref{eq:Phi-bound-v} in Lemma \ref{lem:2.2} to control the contributions from $\Phi_{t_n}(s)$. The proof may then be completed via an application of the continuous form of the Gronwall inequality found for example in \cite[Theorem 1.8.1]{Mao2008}.
\end{proof}

\begin{proof}[Proof of Lemma \ref{lem:4.4}]
On an individual step, $Y$ can be written as:
\begin{equation*}
    Y(u)
= \Phi_{t_n}(u)\left(
    Y(t_n)
    + \int_{t_n}^{u} f\left(\bar Y_n\right)\,ds
    + \int_{t_n}^{u} g\left(\bar Y_n\right)\,dB_2(s)
  \right),
\quad u\in [t_n,t_{n+1}].
\end{equation*}
Using the decomposition \(
\Phi_s(t) = \Phi_s(u)\,\Phi_u(t)
\), we get
\[
\begin{aligned}
Y(t)&-Y(s)\\
=& \Phi_{t_n}(s)\left(\Phi_s(t)-1\right)Y(t_n)+\Phi_{t_n}(s)\left(\Phi_s(t)-1\right)\left(\int_{t_n}^{s} f\left(\bar Y_n\right)\,dr+\int_{t_n}^{s} g\left(\bar Y_n\right)\,dB_2(r)\right) \\
&+\Phi_{t_n}(t)\int_{s}^{t} f\left(\bar Y_n\right)\,dr
+\Phi_{t_n}(t)\int_{s}^{t} g\left(\bar Y_n\right)\,dB_2(r).
\end{aligned}
\]

By the elementary inequality $\left|a_1+a_2+\cdots+a_5\right|^{p}
\le 5^{\,p-1}\left(\,|a_1|^{p}+|a_2|^{p}+\cdots+|a_5|^{p}\right)$, we can estimate
\begin{align}
\mathbb{E}|Y(t)&-Y(s)|^p\nonumber
\\
\le & 5^{p-1}\underbrace{\mathbb{E}\left|\Phi_{t_n}(s)\left(\Phi_s(t)-1\right)Y(t_n)\right|^p}_{=:\tilde A}+5^{p-1}\underbrace{\mathbb{E}\left|\Phi_{t_n}(s)\left(\Phi_s(t)-1\right)\int_{t_n}^{s} f\left(\bar Y_n\right)\,dr\right|^p}_{=:\tilde B}\nonumber\\
&+5^{p-1}\underbrace{\mathbb{E}\left|\Phi_{t_n}(s)\left(\Phi_s(t)-1\right)\int_{t_n}^{s}g\left(\bar Y_n\right)\,dB_2(r)\right|^p}_{=:\tilde C}\nonumber \\
&+ 5^{p-1}\underbrace{\mathbb{E}\left|\Phi_{t_n}(t)\int_{s}^{t} f\left(\bar Y_n\right)\,dr\right|^p}_{=:\tilde D}
+ 5^{p-1}\underbrace{\mathbb{E}\left|\Phi_{t_n}(t)\int_{s}^{t} g\left(\bar Y_n\right)\,dB_2(r)\right|^p}_{=:\tilde E}.\label{eq:boundLemStat}
\end{align}

We will estimate each of the terms $\tilde A$-$\tilde E$ separately. In what follows, let $C>0$ denote a generic positive constant that may change over the course of the proof, but which is always independent of the difference $t-s$.

\underline{Term $\tilde A$:} Observe that since $t_n\leq s\leq t$, $\Phi_{t_n}(s)$, $\Phi_{s}(t)$, and $Y(t_n)$ are pairwise mutually independent. Therefore, by \eqref{eq:Phi1bound} in Lemma \ref{lem:2.1} with $U=Y(t_n)$, \eqref{eq:Phi-bound-v} in Lemma \ref{lem:2.2} with $V=1$, and Lemma \ref{lem:2.3new}, there exists a constant $\bar C>0$ such that 
\begin{align*}
\tilde A=\mathbb{E}\left|\Phi_{t_n}(s)\left(\Phi_s(t)-1\right)Y(t_n)\right|^p
&\leq
\mathbb{E}\left|\Phi_{t_n}(s)\right|^{p}
\mathbb{E}\left|(\Phi_s(t)-1)Y(t_n)\right|^p\\
&\leq C\,(t-s)^{p/2}.
\end{align*}

\underline{Term $\tilde B$:} Recall that $\bar Y_n=Y(t_{n-\lfloor\tau/\delta t\rfloor})$ and $\delta t<\tau$. Then $\Phi_{t_n}(s)$, $\Phi_s(t)$, and $\int_{t_n}^{s} f\left(\bar Y(r)\right)\,dr$ are pairwise mutually independent, and we obtain in a similar way that 
\[
\tilde B=\mathbb{E}\left|\Phi_{t_n}(s)\left(\Phi_s(t)-1\right)\int_{t_n}^{s} f\left(\bar Y_n\right)\,dr\right|^p
    \le C\,(t-s)^{p/2}\mathbb{E}\left|\int_{t_n}^{s} f\left(\bar Y_n\right)\,dr\right|^{p}.
\]
Since $p\geq 1$, the function $|\cdot|^p$ is convex. Apply Jensen's inequality to the last factor on the RHS and bring the expectation inside the integral to get
\begin{equation*}
\mathbb{E}\left|\int_{t_n}^{s} f\left(\bar Y_n\right)\,dr\right|^{p}
\le (s-t_n)^{p-1}\int_{t_n}^{s}\mathbb{E}\left|f\left(\bar Y_n\right)\right|^{p} dr.
\end{equation*}
Applying \eqref{eq:linear-growth} in Assumption \ref{assum:LGcond}, we get
\begin{equation*}
\begin{aligned}
   \tilde B&\le C\,(t-s)^{p/2}\int_{t_n}^{s}\mathbb{E} \left|f\left(\bar Y_n\right)\right|^{p}\,dr\\
    &\le C(t-s)^{p/2}(s-t_n)^{p-1}\int_{t_n}^{s}\mathbb{E}\left(1+\bar{Y}_n^2\right)^{p/2}dr\\
    &\leq C(t-s)^{p/2}.
\end{aligned}
\end{equation*}
At the last step we invoke the bound \eqref{eq:lemma23new} in the statement of Lemma \ref{lem:2.3new}, and the fact that $s-t_n\leq T$.

\underline{Term $\tilde C$:} A similar argument, making additional use of Burkholder-Davis-Gundy inequality (see, for example~\cite[Theorem 1.7.1]{Mao2008}) to bound the absolute moment of the It\^o integral, gives
\[
\begin{aligned}
    \tilde C&=\mathbb{E}\left|\Phi_{t_n}(s)\left(\Phi_s(t)-1\right)\int_{t_n}^{s} g(\bar Y_n)\,dB_2(r)\right|^p
    \\&\le C(t-s)^{p/2}\mathbb{E}\left[\left|\Phi_{t_n}(s)\right|^p\left|\int_{t_n}^{s} g(\bar Y_n)\,dB_2(r)\right|^{p}\right]
    \\& \leq C(t-s)^{p/2}\left(\mathbb{E}\left[\left|\Phi_{t_n}(s)\right|^{2p}\right]\right)^{1/2}\left(\mathbb{E}\left[\left|\int_{t_n}^{s} g(\bar Y_n)\,dB_2(r)\right|^{2p}\right]\right)^{1/2}
    \\ &\leq C(t-s)^{p/2}(s-t_n)^{(p-1)/2}\left(\int_{t_n}^{s}\mathbb{E}\left|g\left(\bar Y_n\right)\right|^{2p}dr\right)^{1/2}
    \\ &\leq C(t-s)^{p/2}(s-t_n)^{(p-1)/2}\left(\int_{t_n}^{s}\mathbb{E}\left(1+\bar{Y}_n^2\right)^{p}dr\right)^{1/2}\\
    & \leq C(t-s)^{p/2}.
\end{aligned}
\]
Note the use of the Cauchy-Schwarz inequality at the second estimation.

\underline{Term $\tilde D$:} Again, \( \Phi_{t_n}(t) \) is independent of \( \int_{s}^{t} f(\bar Y(r))\,dr \), and we may apply Lemma~\ref{lem:2.2} and H\"older's inequality to show that 
\[
\begin{aligned}
\tilde D=\mathbb{E}\left|\Phi_{t_n}(t)\left(\int_{s}^{t} f(\bar Y_n)\,dr\right)\right|^{p}
&= \mathbb{E}|\Phi_{t_n}(t)|^{p}\, \mathbb{E}\left|\int_{s}^{t} f(\bar Y_n)\,dr\right|^{p} \\
&\le C\, \mathbb{E}\left|\left(\int_{s}^{t} \left|f\left(\bar{Y}_n\right)\right|^{p}\,dr\right)^{1/p}(t-s)^{(p-1)/p}\right|^{p}\\
&\le C(t-s)^{p-1}\mathbb{E}\int_{s}^{t}\left|f\left(\bar Y_n\right)\right|^p dr.
\end{aligned}
\]
Next, bringing the expectation inside the integral and applying \eqref{eq:linear-growth}, we get
\begin{align*}
    \tilde D&\le C(t-s)^{p-1}\int_{s}^{t}\mathbb{E}\left(1+\bar{Y}_n^2\right)^{p/2} dr\\
    &\leq C(t-s)^{p}\leq  C(t-s)^{p/2},
\end{align*}
since $t-s<1$.

\underline{Term $\tilde E$:} By Part 2 of Lemma \ref{lem:2.2} we have for $1/p=1/r+1/2$ with $r>1$,
\begin{equation}\label{eq:EtildeBound}
\tilde E=\mathbb{E}\left|\Phi_{t_n}(t)\int_{s}^{t} g(\bar Y_n)\,dB_2(r)\right|^p\leq K_r^{p/r}\left(\mathbb{E}\left|\int_{s}^{t} g(\bar Y_n)\,dB_2(r)\right|^q\right)^{p/q}.
\end{equation}
An application of the Burkholder-Davis-Gundy inequality yields
\[
\mathbb{E}\left|\int_{s}^{t} g(\bar Y_n)\,dB_2(r)\right|^q\leq C(t-s)^{q/2-1}\int_{s}^{t}\mathbb{E}|g(\bar Y_n)|^qdr\leq C(t-s)^{q/2}.
\]

Apply the linear growth bound \eqref{eq:linear-growth} to $g$, and the moment bound \eqref{eq:lemma23new} in Lemma \ref{lem:2.3new}. Then \eqref{eq:Phi-bound-v} in Part 2 of Lemma~\ref{lem:2.2} yields
\[
\tilde E\leq C(t-s)^{p/2}.
\]

Substituting the estimates for terms $\tilde A$-$\tilde E$ back into \eqref{eq:boundLemStat} yields the statement of the Lemma.
\end{proof}

\begin{proof}[Proof of Theorem \ref{thm:main}]
Suppose that the interval $[t_n,t_{n+1}]$ lies on the $m$th delay interval $[m\tau,(m+1)\tau]$. 
The exact solution for $t\in[t_n,t_{n+1}]$ is given by
\[
\begin{aligned}
    X(t)
    =& \Phi_{m\tau}(t)\left[ X(m\tau)
  + \int_{m\tau}^{t} \Phi_{m\tau}^{-1}(s) [f(X(s-\tau))-\rho\sigma g\left(X(s-\tau)\right)] ds \right.\\
  &\left.+ \int_{m\tau}^{t} \Phi_{m\tau}^{-1}(s) g(X(s-\tau)) dB_2(s) \right]\\
  =&\Phi_{m\tau}(t_n)\Phi_{t_n}(t)\left[ X(m\tau)
  + \int_{m\tau}^{t_n} \Phi_{m\tau}^{-1}(s) [f(X(s-\tau))-\rho\sigma g\left(X(s-\tau)\right)] ds\right.\\
  &\left.+ \int_{m\tau}^{t_n} \Phi_{m\tau}^{-1}(s) g(X(s-\tau)) dB_2(s) \right]\\
  &+\Phi_{m\tau}(t_n)\Phi_{t_n}(t)\left[ \int_{t_n}^{t} \Phi_{m\tau}^{-1}(s) f(X(s-\tau)) ds
  + \int_{t_n}^{t} \Phi_{m\tau}^{-1}(s) g(X(s-\tau)) dB_2(s) \right]\\
  =&\Phi_{t_n}(t)\left(X(t_n)+\Phi_{m\tau}(t_n)\left[\int_{t_n}^{t} \Phi_{m\tau}^{-1}(s) [f(X(s-\tau))-\rho\sigma g\left(X(s-\tau)\right)] ds\right.\right.\\
  &\left.\left.+ \int_{t_n}^{t} \Phi_{m\tau}^{-1}(s) g(X(s-\tau)) dB_2(s) \right]\right).
\end{aligned}
\]
Now define the (signed) error function:
\[
\begin{aligned}
    E(t)
:=&X(t)-Y(t)\\
=&\Phi_{t_n}(t)\left(X(t_n)+\Phi_{m\tau}(t_n)\left[\int_{t_n}^{t} \Phi_{m\tau}^{-1}(s) [f(X(s-\tau))-\rho\sigma g\left(X(s-\tau)\right)] ds\right.\right.\\
  &\left.\left.+ \int_{t_n}^{t} \Phi_{m\tau}^{-1}(s) g(X(s-\tau)) dB_2(s) \right]\right.
  \\
  &\left.-Y(t_n)-\int_{t_n}^{t}f\left(Y\left(t_{n-\lfloor\tau/\delta t\rfloor}\right)\right)ds-\int_{t_n}^{t}g\left(Y\left(t_{n-\lfloor\tau/\delta t\rfloor}\right)\right)dB_2(s)\right)\\
  =&\Phi_{t_n}(t)\left(E(t_n)+\int_{t_n}^{t}[\Phi_{t_n}^{-1}(s) [f(X(s-\tau))-\rho\sigma g\left(X(s-\tau)\right)]-f\left(Y\left(t_{n-\lfloor\tau/\delta t\rfloor}\right)\right)]ds\right.\\
  &\left.+\int_{t_n}^{t}[\Phi_{t_n}^{-1}(s) g(X(s-\tau))-g\left(Y\left(t_{n-\lfloor\tau/\delta t\rfloor}\right)\right)]dB_2(s)\right)\\
  =&Z_1(t)Z_2(t),
\end{aligned}
\]
where $Z_1$ and $Z_2$ can be written as components of a system of It\^o-type SDEs with
\[
dZ_1(t)=\mu Z_1(t)dt+\sqrt{1-\rho^2}\sigma Z_1(t)dB_1(t)+\rho\sigma Z_1(t)dB_2(t),
\]
and
\begin{multline*}
dZ_2(t)=\left(\frac{f(X(t-\tau))-\rho\sigma g\left(X(t-\tau)\right)}{Z_1(t)}-f\left(Y\left(t_{n-\lfloor\tau/\delta t\rfloor}\right)\right)\right)dt
\\+\left(\frac{g(X(t-\tau))}{Z_1(t)}-g\left(Y\left(t_{n-\lfloor\tau/\delta t\rfloor}\right)\right)\right)dB_2(t),
\end{multline*}
with initial conditions on the interval $[t_n,t_{n+1}]$ given by $Z_1(t_n)=1$ and $Z_2(t_n)=E(t_n)$. Since the exact solution $X(t)$ admits a finite second moment and,
by Lemma~\ref{lem:2.3new}, $Y(t)$ also admits a finite second moment, it follows that
\[
\mathbb{E}\left[E(t_n)^2\right]
= \mathbb{E}\left[|X(t_n)-Y(t_n)|^2\right]
\le 2\,\mathbb{E}\left[X(t_n)^2\right]
   + 2\,\mathbb{E}\left[Y(t_n)^2\right]
< \infty.
\]
We can use the stochastic product rule to also write $E(t)$ as an It\^{o}-type SDE:
\begin{align*}
dE(t) =&Z_1(t)dZ_2(t)+Z_2(t)dZ_1(t)+dZ_1(t)dZ_2(t)
\\=&\left[ f\left(X(t-\tau)\right) - \Phi_{t_n}(t)\,f\left(Y\left(t_{n-\lfloor\tau/\delta t\rfloor}\right)\right)-\rho\sigma\Phi_{t_n}(t)g\left(Y\left(t_{n-\lfloor\tau/\delta t\rfloor}\right) \right)+ \mu E(t) \right] dt\\
&+\sqrt{1-\rho^2}\sigma E(t)  dB_1(t)\\
&+\left[ g\left(X(t-\tau)\right)
        - \Phi_{t_n}(t)\,g\left(Y\left(t_{n-\lfloor\tau/\delta t\rfloor}\right)\right)+\rho\sigma E(t)\right]dB_2(t).
\end{align*}
We may then apply It\^o's formula to compute the SDE governing $E(t)^2$, 
write in integral form for any $t \in [t_n, t_{n+1}]$, and take expectations on both sides, 
%
making use of the martingale property to eliminate the It\^o integrals:
\begin{align}
    \mathbb{E}[E(t)^2]
    =&\mathbb{E}\left[E(t_n)^2\right]+ (2\mu+\sigma^2)\int_{t_n}^{t} \mathbb{E}\left[E(s)^2\right]ds\nonumber\\
    &+2\int_{t_n}^{t}\mathbb{E}\left[\underbrace{\left(f(X(s-\tau))-\Phi_{t_n}(s)f\left(Y\left(t_{n-\lfloor\tau/\delta t\rfloor}\right)\right)\right)E(s)}_{=:\bar A}\right]ds\nonumber\\
    &+\int_{t_n}^{t}\mathbb{E}\left[\underbrace{\left(g(X(s-\tau))-\Phi_{t_n}(s) g\left(Y\left(t_{n-\lfloor\tau/\delta t\rfloor}\right)\right)\right)^2}_{=:\bar B} \right]ds\nonumber\\
   &-2\rho\sigma\int_{t_n}^{t}\mathbb{E}\left[\underbrace{\Phi_{t_n}(s)g\left(Y\left(t_{n-\lfloor\tau/\delta t\rfloor}\right)\right)E(s)}_{=:\bar C}\right]ds.\label{eq:sdeE2}\end{align}
We can rewrite the term $\bar A$ as follows:
\begin{align*}
\bar A=&\left(f(X(s-\tau))-\Phi_{t_n}(s)f\left(Y\left(t_{n-\lfloor\tau/\delta t\rfloor}\right)\right)\right)E(s)\\
=& \underbrace{E(s)\left(f(X(s-\tau))-f(Y(s-\tau))\right)}_{=:\bar A_1}+ \underbrace{E(s)\left(f(Y(s-\tau))-\Phi_{t_n}(s)\,f\left(Y(s-\tau)\right)\right)}_{=:\bar A_2}\\
&+\underbrace{E(s)\Phi_{t_n}(s)\left(f(Y(s-\tau))-\,f\left(Y\left(t_{n-\lfloor\tau/\delta t\rfloor}\right)\right)\right)}_{=:\bar A_3}.  
\end{align*}
We can rewrite the term $\bar B$ as follows:
\begin{align*}
\bar B=&\left[g(X(s-\tau))-\Phi_{t_n}(s) g\left(Y\left(t_{n-\lfloor\tau/\delta t\rfloor}\right)\right) \right]^2\\
=& \left[\underbrace{\left(g(X(s-\tau))-g(Y(s-\tau))\right)}_{=:\bar B_1}+ \underbrace{\left(g(Y(s-\tau))-\Phi_{t_n}(s)\,g\left(Y(s-\tau)\right)\right)}_{=:\bar B_2}\right.\\
&\left.+\underbrace{\Phi_{t_n}(s)\left(g(Y(s-\tau))-\,g\left(Y\left(t_{n-\lfloor\tau/\delta t\rfloor}\right)\right)\right)}_{=:\bar B_3}\right]^2.  
\end{align*}
We bound each of the terms $\bar A_{i},\,\bar B_i$, $i=1,2,3$ separately. As in the proof of Lemma \ref{lem:4.4}, $C$ will represent a generic constant that may change in value over the course of the proof, but which is always independent of $\delta t$.

\underline{Term $\bar A_1$:} By \eqref{eq:4.5} in Assumption~\ref{ass:delayOSL}, there exists a constant $C>0$ such that
\[
\bar{A}_1= E(s)\left(f(X(s-\tau)) - f(Y(s-\tau))\right)\leq C E(s)^{2},\quad \text{a.s.}
\]
This bound applies pathwise and therefore
\begin{equation}\label{eq:A1bound}
\mathbb{E}\left[\bar{A}_1\right]\leq C\mathbb{E}\left[E(s)^{2}\right].
\end{equation}

\underline{Term $\bar A_2$:} Apply the elementary inequality $ab\leq a^{2}/2 + b^{2}/2$ to get
\begin{align*}
  \bar{A}_2 &= E(s) f(Y(s-\tau))(1-\Phi_{t_n}(s))\le \frac{1}{2}  E(s)^2 + \frac{1}{2} f(Y(s-\tau))^2 (1-\Phi_{t_n}(s))^2,\quad a.s.  
\end{align*}
Take expectations on both sides:
\begin{align}
    \mathbb{E}\left[\bar A_2\right]&\leq\frac{1}{2}\mathbb{E}\left[E(s)^2\right]+\frac{1}{2}\mathbb{E}\left[f(Y(s-\tau))^2 (1-\Phi_{t_n}(s))^2\right]\nonumber\\
    &\leq\frac{1}{2}\mathbb{E}\left[E(s)^2\right]+\frac{1}{2}\mathbb{E}\left[f(Y(s-\tau))^2\right]\mathbb{E}\left[(1-\Phi_{t_n}(s))^2\right]\nonumber\\
    &\leq \frac{1}{2}\mathbb{E}\left[E(s)^2\right]+C\delta t.\label{eq:A2bound}
\end{align}
At the second step we use the fact that since $s-\tau<t_n$, the terms $f\left(Y(s-\tau)\right)$ and $1-\Phi_{t_n}(s)$ are mutually independent. At the third step we apply the moment bounds in Lemmas~\ref{lem:2.1} and \ref{lem:2.3new}.

\underline{Term $\bar A_3$:} 
Apply the same elementary inequality to get
\[
\bar{A}_3\le \frac{1}{2}  E(s)^2 + \frac{1}{2} \Phi_{t_n}(s)^2\left(f(Y(s-\tau))-\,f\left(Y\left(t_{n-\lfloor\tau/\delta t\rfloor}\right)\right)\right)^2\quad a.s.
\]
By the independence of $\Phi_{t_n}(s)^{2}$ and $\left(f(Y(s-\tau))-\,f\left(Y\left(t_{n-\lfloor\tau/\delta t\rfloor}\right)\right)\right)^2$, and taking expectations on both sides,
\begin{align}
\mathbb{E}\left[\bar{A}_3\right]&\leq \frac{1}{2}  \mathbb{E}\left[E(s)^2\right]
+ \frac{1}{2} \mathbb{E}\left[\Phi_{t_n}(s)^2\right]\mathbb{E}\left[\left(f(Y(s-\tau))-f\left(Y\left(t_{n-\lfloor\tau/\delta t\rfloor}\right)\right)\right)^2\right]\nonumber\\
&\leq  \frac{1}{2}  \mathbb{E}\left[E(s)^2\right]+\frac{1}{2}C\mathbb{E}\left[\left(f(Y(s-\tau))-f\left(Y\left(t_{n-\lfloor\tau/\delta t\rfloor}\right)\right)\right)^2\right]\nonumber\\
&\leq  \frac{1}{2}  \mathbb{E}\left[E(s)^2\right]+\frac{1}{2}C\mathbb{E}\left[\left(Y(s-\tau)-Y\left(t_{n-\lfloor\tau/\delta t\rfloor}\right)\right)^2\right]\nonumber\\
&\leq \frac{1}{2}  \mathbb{E}\left[E(s)^2\right]+\frac{1}{2}C\delta t, \label{eq:A3bound}
\end{align}
where we use Lemma~\ref{lem:2.2} at the second step, \eqref{eq:linear-growth} in Assumption~\ref{assum:LGcond} at the third step, and Lemma~\ref{lem:4.4} at the last step.

By combining the estimates \eqref{eq:A1bound}, \eqref{eq:A2bound}, and \eqref{eq:A3bound}, we arrive at the overall estimate
\begin{equation}\label{eq:barAbound}
\mathbb{E}\left[\bar A\right]\leq C(\mathbb{E}\left[E(s)^2\right]+\delta t).    
\end{equation}

Next, we have $(\bar{B}_1+\bar{B}_2+\bar{B}_3)^2\leq 3\bar{B}_1^2+3\bar{B}_2^2+3\bar{B}_3^2$. We estimate term-by-term.

\underline{Term $\bar B_1^2$:} By \eqref{eq:gLipschitz} in Assumption~\ref{ass:delayOSL}, there exists a constant $C>0$ such that
\begin{equation}\label{eq:bare2}
\mathbb{E}\left[\bar{B}_1^2\right] \le C\mathbb{E}\left[E(s-\tau)^{2}\right].
\end{equation}

\underline{Term $\bar B_2^2$:} By the independence of $(1-\Phi_{t_n}(s))^{2}$ and $g\left(Y(s-\tau)\right)^2$, Lemma~\ref{lem:2.1} and Lemma~\ref{lem:2.2}, there exists a constant $C>0$ such that
\begin{equation}\label{eq:barf2}
 \mathbb{E}\left[\bar{B}_2^2\right] =\mathbb{E}\left[g(Y(s-\tau))^2(1-\Phi_{t_n}(s))^2\right]\leq C\delta t .
\end{equation}

\underline{Term $\bar B_3^2$:} 
By the independence of $\Phi_{t_n}(s)^{2}$ and $\left(g(Y(s-\tau))-\,g\left(Y\left(t_{n-\lfloor\tau/\delta t\rfloor}\right)\right)\right)^2$, applying Lemma~\ref{lem:2.2}, the bound \eqref{eq:gLipschitz} in Assumption \ref{ass:delayOSL}, and Lemma~\ref{lem:4.4}, there exists a constant $C>0$ such that
\begin{align}
 \mathbb{E}\left[\bar{B}_3^2\right]=&\mathbb{E}\left[\Phi_{t_n}(s)^2\left(g(Y(s-\tau)\right)-g\left(Y\left(t_{n-\lfloor\tau/\delta t\rfloor}\right)\right)^2\right]\nonumber\\
 =&\mathbb{E}\left[\Phi_{t_n}(s)^2\right]\mathbb{E}\left[\left(g(Y(s-\tau)\right)-g\left(Y\left(t_{n-\lfloor\tau/\delta t\rfloor}\right)\right)^2\right]\nonumber\\
 \le& C\delta t.\label{eq:barg2}
\end{align}
By combining the estimates \eqref{eq:bare2}, \eqref{eq:barf2}, and \eqref{eq:barg2}, we arrive at the overall estimate
\begin{equation}\label{eq:barBbound}
\mathbb{E}\left[\bar B\right]\leq C\left(\mathbb{E}\left[E(s-\tau)^2\right]+\delta t\right),  
\end{equation}
for some $C>0$.

\underline{Term $\bar C$:}

By the independence of $\Phi_{t_n}(s)^{2}$ and $g\left(Y\left(t_{n-\lfloor\tau/\delta t\rfloor}\right)\right)^2$, as well as the bounds given in  Lemmas~\ref{lem:2.2} and~\ref{lem:2.3new}, there exists a constant $C_{12}>0$ such that
\begin{equation}\label{eq:bard}
    -2\rho\sigma\mathbb{E}\left[\bar{C}\right] \le|\rho|C\left(\mathbb{E}\left[E(s)^{2}\right] + 1\right).
\end{equation}
Note that we have kept the linear dependence on $|\rho|$ explicit in this bound, rather than absorbing it into the constant $C$.

Substituting the bounds \eqref{eq:barAbound}, \eqref{eq:barBbound}, and \eqref{eq:bard}, back into \eqref{eq:sdeE2} and subtracting $\mathbb{E}\left[E(t_n)^2\right]$ from both sides gives
\begin{align}
\mathbb{E}&\left[E(t)^2\right] - \mathbb{E}\left[E(t_n)^2\right]\nonumber\\
&\leq(2\mu+\sigma^2)\int_{t_n}^{t} \mathbb{E}\left[E(s)^2\right] ds+C\int_{t_n}^{t}\left(\mathbb{E}\left[E(s)^2\right]+\mathbb{E}\left[E(s-\tau)^2\right]+\delta t+C|\rho|\right)ds\nonumber\\
&\leq C\int_{t_n}^{t} \mathbb{E}\left[E(s)^2\right] ds+C\int_{t_n}^{t}\mathbb{E}\left[E(s-\tau)^2\right]ds+C\delta t^2+C|\rho|\delta t.\label{eq:preGronwall}
\end{align}
Let $\bar{N} = \max\{n : t_n \leq t\}$ be the index of the meshpoint immediately preceding the time $t$ in \eqref{eq:preGronwall}. Now set $t = t_{\bar{N}}$ and sum both sides over $n = 0, \ldots, \bar{N}$. Finally, since \eqref{eq:preGronwall} 
also holds if we replace $t_n$ with $t_{\bar{N}}$ we can add it again to both sides 
yielding the inequality
\begin{equation*}
\mathbb{E}\left[E(t)^2\right]
\le C\int_{0}^{t} \mathbb{E}\left[E(s)^2\right] \, ds+C\int_{0}^{t} \mathbb{E}\left[E(s-\tau)^2\right] \, ds
+ CT\delta t +C|\rho|T.
\end{equation*}
We have used the fact that $E(0) = 0$. Since this inequality holds for all 
$t \in [0,T]$, an application of the delay form of the Gronwall inequality (see Lemma ~\ref{le:delaygrownwall}) yields the statement of the Theorem.
\end{proof}

\medskip
\noindent {\bf Acknowledgement:}

This article is based upon work from COST Action
CA24104 Stochastic Differential Equations: Computation, Inference, Applications (STOCHASTICA), supported by COST (European Cooperation in Science and Technology). Website: \url{www.cost.eu}. The second author is supported by the China Scholarship Council (CSC).

The authors wish to thank Prof. Evelyn Buckwar, Johannes Kepler University, for valuable discussion and guidance at the outset of this project while hosting a visit by the first author in January 2023, and in the context of a Royal Society of Edinburgh Saltire Research Facilitation Network (Ref. 1832) over 2022-24.  

We also wish to thank the anonymous referee for their careful reading of the manuscript and suggestions, which have considerably improved the final version.

\bibliographystyle{plain}

\end{document}